\documentclass{article}

\usepackage{amsmath}
\usepackage{amsfonts}
\usepackage{amssymb}
\usepackage{amsthm}

\newcommand{\R}{\mathbb{R}}
\newcommand{\N}{\mathbb{N}}

\newcommand{\bs}{\boldsymbol}
\newcommand{\mc}{\mathcal}
\newcommand{\E}{\mathbb{E}}

\renewcommand{\Pr}{\mathbb{P}}

\theoremstyle{definition}

\newtheorem{corollary}{Corollary}
\newtheorem{lemma}{Lemma}
\newtheorem{proposition}{Proposition}

\newtheorem{definition}{Definition}
\newtheorem{example}{Example}

\theoremstyle{remark}

\begin{document}

\title{Constant Rate Distributions on Partially Ordered Sets}

\author{Kyle Siegrist\\
Department of Mathematical Sciences \\
University of Alabama in Huntsville\\
siegrist@math.uah.edu}

\maketitle

\begin{abstract}
We consider probability distributions with constant rate on partially ordered sets, generalizing distributions in the usual reliability setting $([0, \infty), \le)$ that have constant failure rate.  In spite of the minimal algebraic structure, there is a surprisingly rich theory, including
moment results and results concerning ladder variables and point processes.  We concentrate mostly on discrete posets, particularly posets whose graphs are rooted trees.  We pose some questions on the existence of constant rate distributions for general discrete posets.

\vspace{5 pt}

\vspace{5 pt}
\end{abstract}

\section{Preliminaries}

\subsection{Introduction}

The exponential distribution on $[0, \infty)$ and the geometric distribution on $\N$ are characterized by the {\em constant rate} property:  the density function (with respect to Lebesgue measure in the first case and counting measure in the second) is a multiple of the upper (right-tail) distribution function.  

The natural mathematical home for the constant rate property is a partially ordered set (poset) with a reference measure for the density functions.  In this paper we explore these distributions. In spite of the minimal algebraic structure of a poset, there is a surprisingly rich theory, including moment results and results concerning ladder variables and point processes.  In many respects, constant rate distributions lead to the most random way to put ordered points in the poset.  We will be particularly interested in the existence question---when does a poset support constant rate distributions?

\subsection{Standard posets}

Suppose that $(S, \preceq)$ is a poset. For $x \in S$, let
\[
	I(x) = \{t \in S \colon t \succ x\}, \; I[x] = \{t \in S \colon t \succeq x\}, \; D[x] = \{t \in S \colon t \preceq x\}
\]
For $n \in \N$ and $x \in S$ we also define
\begin{align*}
	D_n &= \{(x_1, x_2, \ldots, x_n) \in S^n: x_1 \preceq x_2 \preceq \cdots \preceq x_n\} \\
	D_n[x] &= \{(x_1, x_2, \ldots, x_n) \in S^n: x_1 \preceq x_2 \preceq \cdots \preceq x_n \preceq x \}
\end{align*}
Recall also that $A \subseteq S$ is {\em convex} if 
\[
a \in A, \, b \in A, \, a \preceq x \preceq b \Rightarrow x \in A
\]
For $x, \, y \in S$, $y$ is said to {\em cover} $x$ if $y$ is a minimal element of $I(x)$. If $S$ is countable, the {\em Hasse graph} or {\em covering graph} of $(S, \preceq)$ has vertex set $S$ and (directed) edge set $\{(x, y) \in S^2 : y \text { covers } x\}$. We write $x \perp y$ if $x \preceq y$ or $y \preceq x$ and we write $x \parallel y$ if neither $x \preceq y$ nor $y \preceq x$. The poset $(S, \preceq)$ is {\em connected} if for every $x, \, y \in S$, there exists a finite sequence $(x_0, x_1, \ldots, x_n)$ such that $x_0 = x$, $x_n = y$, and $x_{i-1} \perp x_i$ for $i = 1, \ldots, n$.

We assume that $S$ has a topology that is compatible with the algebraic structure, in the sense of the following properties
\begin{enumerate}
	\item $S$ is locally compact, locally convex, and has a countable base.
	\item $D_2$ is closed in $S^2$ (with the product topology).
	\item $D[x]$ is compact for each $x \in S$.
\end{enumerate}
The Borel $\sigma$-algebra will be denoted by $\mathcal{B}(S)$.  The first assumptions means that the partially ordered set $(S, \, \preceq)$ is topologically closed in the sense of Nachbin \cite{Nachbin}.  In particular, this assumption implies that $I[x]$ and $D[x]$ are closed for each $x$, and that $S$ is Hausdorff. Hence, if $S$ is countable, then $\mathcal{B}(S)$ is the power set of $S$.  Conversely, suppose that $(S, \, \preceq)$ is a poset and that $S$ is countable.  If $D[x]$ is finite for each $x \in S$ (so that $(S, \, \preceq)$ is {\em locally finite}), then with the discrete topology, $(S, \, \preceq)$ satisfies the topological assumptions.

Next, we fix a positive Borel measure $\lambda$ on $\mathcal{B}(S)$ as a {\em reference measure}. Thus $\lambda(K) < \infty$ for compact $K \subseteq S$. We also assume that $\lambda$ that has support $S$, so that $\lambda(U) > 0$ for every open set $U$. Because of the topological assumptions, $\lambda$ is $\sigma$-finite and is regular. When $S$ is countable, we always take the reference measure to be counting measure.

\begin{definition}
The triple $(S, \preceq, \lambda)$ satisfying the algebraic, topological, and measure-theoretic assumptions will be called a {\em standard poset}.
\end{definition}

An important special case is the poset associated with a {\em positive semigroup}.  A positive semigroup $(S, \cdot)$ is a semigroup that has an identity element $e$, satisfies the left-cancellation law, and has no non-trivial inverses.  The partial order $\preceq$ associated with $(S, \cdot)$ is given by
\[
x \preceq y \text{ if and only if } xt = y \text{ for some } t \in S
\]
If $x \preceq y$, then $t$ satisfying $x t = y$ is unique, and is denoted $x^{-1} y$. In addition to the topological assumptions given above, we assume that $(x, y) \to xy$ is continuous.  We also assume that the reference measure $\lambda$ is {\em left-invariant}, that is 
\[
\lambda(xA) = \lambda(A), \quad x \in S, \, A \in \mc{B}(S)
\]
The canonical examples are $([0, \infty), +)$ with $\lambda$ as Lesbesgue meausre, and $(\N, +)$ (with counting measure of course).  In both cases, the associated order is the ordinary order $\le$. Probability distributions (particularly exponential type distributions) on positive semigroups have been studied in \cite{Rowell, Siegrist, Siegrist2, Siegrist3, Siegrist4}. The critical feature of a positive semigroup is that for each $x$, $I[x] = xS$ looks like the entire space $S$ from algebraic, topological, and measure-theoretic points of view.  Of course, there is no such self-similarity for general standard posets.

For the remainder of this article, unless otherwise noted, we assume that $(S, \preceq, \lambda)$ is a standard poset.

\subsection{Operators and cumulative functions}

Let $\mc{D}(S)$ denote the set of measurable functions from $S$ into $\R$ that are bounded on the compact set $D[x]$ for each $x \in S$. Define the {\em lower operator} $L$ on $\mc{D}(S)$ by
\[
Lf(x) = \int_{D[x]} f(t) d\lambda(t), \quad x \in S
\]
Next, let $\mc{L}(S)$ denote the usual Banach space of measurable functions $f: S \to \R$, with $||f|| = \int_S |f(x)| d\lambda(x) < \infty$. Define  the {\em upper operator} $U$ on $\mc{L}(S)$ by
\[
Uf(x) = \int_{I[x]} f(y) d\lambda(y), \quad x \in S
\]
A simple application of Fubinni's theorem gives the following duality relationship between the linear operators $L$ and $U$:
\begin{equation}
\int_S Lf(x) g(x) d\lambda(x) = \int_S f(x) Ug(x) d\lambda(x) \label{eq.dual}
\end{equation}
assuming, of course, the appropriate integrability conditions. Both operators can be written as integral operators with a kernel function.  Define $r: S \times S \to \R$ by 
\[r(x, y) = \begin{cases}
1 &\text{ if } x \preceq y \\
0 &\text{ otherwise}
\end{cases}\]
Then
\begin{align*}
Lf(x) &= \int_S r(t, x) f(t) d\lambda(t), \quad x \in S\\
Uf(x) &= \int_S r(x, t) f(t) d\lambda(t), \quad x \in S
\end{align*}
In the discrete case, $r$ is the {\em Riemann function} in the terminology of M\"obius inversion \cite{Berge}, and its inverse $m$ (also in the sense of this theory) is the {\em M\"obius function}.  The lower operator $L$ is invertible, and if $g = Lf$ then
\[
f(x) = \sum_{t \in S} g(t) m(t, x)
\]
As we will see in Example \ref{ex.nonUnique}, the upper operator $U$ is not invertible in general, even in the discrete case.

Now let $\bs{1}$ denote the constant function $1$ on $S$, and define $\lambda_n: S \to [0, \infty)$ by $\lambda_n = L^n\bs{1}$ for $n \in \N$. Equivalently, $\lambda_n(x) = \lambda^n(D_n[x])$ for $n \in \N_+$ and $x \in S$, where $\lambda^n$ is $n$-fold product measure on $S^n$. We will refer to $\lambda_n$ as the {\em cumulative function of order $n$}; these functions play an important role in the study of probability distribution on $(S, \preceq)$. For the poset $([0, \infty), \le)$, the cumulative functions are
\[
\lambda_n(x) = \frac{x^n}{n!}, \quad x \in [0, \infty), \, n \in \N
\]
For the poset $(\N, \le)$, the cumulative functions are
\begin{equation}
\lambda_n(x) = \binom{n+x}{x}, \quad x, n \in \N \label{eq.cumulativeN}
\end{equation}
Other examples, for positive semigroups, are given in \cite{Siegrist4}.

The ordinary generating function of $n \mapsto \lambda_n(x)$ is the function $\Lambda$ given by
\[ \Lambda(x, t) = \sum_{n=0}^\infty \lambda_n(x) t^n \]
for $x \in S$ and for $t \in \R$ for which the series converges absolutely. The generating function $\Lambda$ is important in the study of the point process associated with a constant rate distribution. For the poset $([0, \infty), \le)$, $\Lambda(x, t) = e^{t x}$. For the poset $(\N, \le)$, 
\begin{equation} \label{eq.GF2}
\Lambda(x, t) = \frac{1}{(1 - t)^{x + 1}}, \quad x \in \N, \, |t| < 1
\end{equation}

\section{Probability Distributions}

\subsection{Distribution functions}

Suppose now that $X$ is a random variable taking values in $S$. We will assume that the distribution of $X$ is a Borel probability measure and that $X$ has support $S$. The probability density function (PDF) of $X$, if it exists, refers to the density with respect the reference measure $\lambda$.

\begin{definition}
The {\em upper probability function} (UPF) of $X$ is the function $F: S \to (0, 1]$ given by
\[
F(x) = \Pr(X \succeq x) = \Pr(X \in I[x]), \quad x \in S
\]
\end{definition}

If $X$ has a PDF $f$ (which is always the case when $S$ is discrete) then
\[
F(x) = Uf(x) = \int_{I[x]} f(t) d\lambda(t), \quad x \in S
\]
In general, the UPF of $X$ does not uniquely determine the distribution of $X$. 

\begin{example} \label{ex.nonUnique}
Let $A$ be fixed set with $k$ elements ($k \geq 2$), and let $(S, \preceq)$ denote the lexicographic sum of the anti-chains $(A_n, =)$ over $(\N, \leq)$, where $A_0 = \{e\}$ and $A_n = A$ for $n \in \N_+$. Thus, $(0, e)$ is the minimum element of $S$ and for $n, \, m \in \N_+$ and $x, \, y \in A$, $(m, x) \prec (n, y)$ if and only if $n < m$. 

Now let $f$ be a PDF on $S$ with UPF $F$.  Define $g$ by 
\[ g(n, x) = f(n, x) + \left(-\frac{1}{k-1}\right)^n c \]
where $c$ is a constant. It is straightforward to show that
\[
\sum_{(m,y) \succeq (n,x)} g(m,y) = F(n, x), \quad (n, x) \in S
\]
In particular, $\sum_{(x, n) \in S} g(n, x) = 1$.  Hence if we can choose a PDF $f$ and a nonzero constant $c$ so that $g(n, x) > 0$ for every $(n, x) \in S$, then $g$ is a PDF different from $f$ but with the same UPF. This can always be done when $k \ge 3$. 
\end{example}

In the discrete case, a simple application of the inclusion-exclusion rule shows that the distribution of $X$ {\em is} determined by the generalized UPF, defined for finite $A \subseteq S$ by 
\[
F(A) = \Pr(X \succeq x \text{ for all } x \in A)
\]
An interesting problem is to give conditions on the poset $(S, \preceq)$ that ensure that a distribution on $S$ is determined by its ordinary UPF.  This holds for a discrete upper semilattice, since $F(A) = F(\sup(A))$ for $A$ finite.  It also holds for trees, as we will see in Section \ref{s.trees}.

\begin{definition}
Suppose that $X$ has UPF $F$ and PDF $f$. The {\em rate function} of $X$ is the function $r \colon S \to (0, \infty)$ defined by
\[
r(x) = \frac{f(x)}{F(x)}, \quad x \in S
\]
\end{definition}

If the poset is $([0,\,\infty), \leq)$, the rate function is the ordinary {\em failure rate function}; the value at $x$ gives the probability density of failure at $x$ given survival up to $x$. Of course, in the general setting of posets, the reliability interpretation does not have its usual meaning; we are using ``rate'' simply as a convenient mathematical term. Note that for a discrete poset, $r(x) = \Pr(X = x | X \succeq x)$. In general, $X$ has {\em constant rate} if $r$ is constant on $S$, {\em increasing rate} if $r$ is increasing on $S$ and {\em decreasing rate} if $r$ is decreasing on $S$. We are particularly interested in distributions with constant rate; these will be studied in Section \ref{s.rate}. 

\begin{proposition} \label{p.rate0}
Suppose that $(S, \preceq)$ is a standard discrete poset, and that $X$ has rate function $r$.  Then $r(x) \leq 1$ for $x \in S$ and $r(x) = 1$ if and only if $x$ is maximal.
\end{proposition}

\begin{proof}
Let $f$ denote the PDF of $X$ and $F$ the UPF. For $x \in S$,
\[
F(x) = f(x) + \Pr(X \succ x) = r(x) F(x) + \Pr(X \succ x)
\]
Hence
\begin{equation} \label{eq.rate0}
\Pr(X \succ x) = [1 - r(x)]F(x)
\end{equation}
By (\ref{eq.rate0}), $r(x) \leq 1$ for all $x \in S$. If $x$ is maximal, then $\Pr(X \succ x) = \Pr(\emptyset) = 0$ so $r(x) = 1$. Conversely, if $x$ is not maximal then $\Pr(X \succ x) > 0$ (since $X$ has support $S$) and hence $r(x) < 1$.
\end{proof}

The following proposition gives a simple result that relates expected value to the lower operator $L$. For positive semigroups, this result was given (in different notation) in \cite{Siegrist4}.

\begin{proposition} \label{p.expect}
Suppose that $X$ has UPF $F$. For $g \in \mc{D}(S)$ and $n \in \N$,
\begin{equation}
\int_S L^ng(x) F(x) d\lambda(x) = \E[L^{n+1}g(X)] \label{eq.expect}
\end{equation}
\end{proposition}

\begin{proof} 
Using Fubinni's theorem,
\begin{align*}
	\int_S L^ng(x) \Pr(X \succeq x) d\lambda(x) &= \int_S L^ng(x) \E[\bs{1}(X \succeq x)] d\lambda(x) \\
	&= \E \left(\int_{D[X]} L^n(g)(x) d\lambda(x) \right) \\
	&= \E[L^{n+1}g(X)]
\end{align*}
\end{proof}

When $X$ has a PDF $f$, (\ref{eq.expect}) also follows from (\ref{eq.dual}). In particular, letting $g = \bs{1}$ gives
\begin{equation} \label{eq.expect1}
\int_S \lambda_n(x) F(x) d\lambda(x) = \E[\lambda_{n+1}(X)], \quad n \in \N
\end{equation}
and when $n = 0$, (\ref{eq.expect1}) becomes
\begin{equation} \label{eq.expect2}
\int_S F(x) d\lambda(x) = \E(\lambda (D[X]))
\end{equation}
If the poset is $(\mathbf{N}, \le)$, then (\ref{eq.expect2}) reduces to the standard result
\[
\sum_{n=0}^\infty \Pr(X \ge n) = \E(X) + 1
\]
If the poset is $([0, \, \infty), \, \le)$, then (\ref{eq.expect2}) reduces to the standard result
\[
\int_0^\infty \Pr(X \ge x) \, dx = \E(X)
\]

\subsection{Ladder variables and partial products} \label{ss.ladder1}

Let $\bs{X} = (X_1, X_2, \ldots)$ be a sequence of independent, identically distributed random variables, taking values in $S$, with common UPF $F$ and PDF $f$. We define the sequence of {\em ladder variables} $\bs{Y} = (Y_1, Y_2, \ldots)$ associated with $\bs{X}$ as follows: First let
\[
N_1 = 1, \; Y_1 = X_1
\]
and then recursively define
\[
N_{n+1} = \min\{n > N_n: X_n \succeq Y_n\}, \; Y_{n+1} = X_{N_{n+1}}
\]

\begin{proposition} \label{p.ladder1}
The sequence $\bs{Y}$ is a homogeneous Markov chain with transition density $g$ given by
\[
g(y, z) = \frac{f(z)}{F(y)}, \quad (y, z) \in D_2
\]
\end{proposition}

\begin{proof}
Let $(y_1, \ldots, y_{n-1}, y, z) \in D_{n+1}$.  The conditional distribution of $Y_{n+1}$ given $\{Y_1 = y_1, \ldots, Y_{n-1} = y_{n-1}, Y_n = y\}$ corresponds to observing independent copies of $X$ until a variable occurs with a value greater that $y$ (in the partial order).  The distribution of this last variable is the same as the conditional distribution of $X$ given $X \succeq y$, which has density $z \mapsto f(z) / F(y)$ on $I[y]$.
\end{proof}

Since $Y_1 = X_1$ has PDF $f$, it follows immediately from Proposition \ref{p.ladder1} that $(Y_1, Y_2, \ldots, Y_n)$ has PDF $g_n$ (with repsect to $\lambda^n$) given by
\[
g_n(y_1, y_2, \ldots, y_n) = f(y_1) \frac{f(y_2)}{F(y_1)} \cdots \frac{f(y_n)}{F(y_{n-1})}, \quad (y_1, y_2, \ldots, y_n) \in D_n
\]
This PDF has a simple representation in terms of the rate function $r$:
\begin{equation}
g_n(y_1, y_2, \ldots, y_n) = r(y_1) r(y_2) \cdots r(y_{n-1}) f(y_n), \quad (y_1, y_2, \ldots, y_n) \in D_n \label{eq.jointPDF}
\end{equation}

Suppose now that $(S, \cdot, \lambda)$ is a standard positive semigroup, and that $\bs{X} = (X_1, X_2, \ldots)$ is an IID sequence in $S$ with PDF $f$. Let $Z_n = X_1 \cdots X_n$ for $n \in \N_+$, so that $\bs{Z} = (Z_1, X_2 \ldots)$ is the {\em partial product} sequence associated with $\bs{Z}$.

\begin{proposition} \label{p.product1}
The sequence $\bs{Z}$ is a homogeneous Markov chain with transition probability density $h$ given by
\[
h(y, z) = f(y^{-1} z), \quad (y, z) \in D_2
\]
\end{proposition}

Since $Z_1 = X_1$ has PDF $f$, it follows immediately from Proposition \ref{p.product1} that $(Z_1, Z_2, \ldots, Z_n)$ has PDF $h_n$ (with respect to $\lambda^n$) given by
\[
h_n(z_1, z_2, \ldots, z_n) = f(z_1) f(z_1^{-1} z_2) \cdots f(z_{n-1}^{-1} z_n), \quad (y_1, y_2, \ldots, y_n) \in D_n
\]
So in the case of a standard positive semigroup, there are two natural processes associated with an IID sequence $\bs{X}$: the sequence of ladder variables $\bs{Y}$ and the partial product sequence $\bs{Z}$.  In general, these sequences are not equivalent, but as we will see in Section \ref{s.rate}, {\em are} equivalent when the underlying distribution of $\bs{X}$ has constant rate. Moreover, this equivalence characterizes constant rate distributions.

Return now to the setting of a standard poset $(S, \preceq \lambda)$. If $\bs{W} = (W_1, W_2, \ldots)$ is an increasing sequence of random variables in $S$ (such as a sequence of ladder variables, or in the special case of a positive semigroup, a partial product sequence), we can construct a point process in the usual way. For $x \in S$, let
\[
N_x = \#\{n \in \N_+: W_n \preceq x\}
\]
so that $N_x$ is the number of random points in $D[x]$. We have the usual inverse relation between the processes $\bs{W}$ and $\bs{N} = (N_x: x \in S)$, namely $W_n \preceq x$ if and only if $N_x \geq n$ for $n \in \N_+$ and $x \in S$. For $n \in \N_+$, let $G_n$ denote the {\em lower} probability function of $W_n$, so that $G_n(x) = \Pr(W_n \preceq x)$ for $x \in S$. Then $\Pr(N_x \geq n) = \Pr(W_n \preceq x) = G_n(x)$ for $ n \in \N_+$. Of course, $\Pr(N_x \geq 0) = 1$.  If we define $G_0(x) = 1$ for all $x \in S$, then for fixed $x$, $n \mapsto G_n(x)$ is the {\em upper} probability function of $N_x$. 

The sequence of random points $\bs{W}$ can be {\em thinned} in the usual way. Specifically, suppose that each point is accepted with probability $p \in (0, 1)$ and rejected with probability $1 - p$, independently from point to point.  Then the first excepted point is $W_M$ where $M$ is independent of $\bs{W}$ and has the geometric distribution on $\N_+$ with parameter $p$.

\section{Distributions with Constant Rate} \label{s.rate}

\subsection{Characterizations and properties}

Suppose that $X$ is a random variable taking values in $S$ with UPF $F$. Recall that $X$ has {\em constant rate} $\alpha > 0$ if $f = \alpha F$ is a PDF of $X$. 

If $(S, \preceq)$ is associated with a positive semigroup $(S, \cdot)$, then $X$ has an {\em exponential} distribution if 
\begin{equation*}
\Pr(X \in xA) = F(x) \Pr(X \in A), \quad x \in S, A \in \mc{B}(S)
\end{equation*}
Equivalently, the conditional distribution of $x^{-1}X$ given $X \succeq x$ is the same as the distribution of $X$.  Exponential distributions on positive semigroups are studied in \cite{Rowell, Siegrist, Siegrist2, Siegrist3, Siegrist4}. In particular, it is shown in \cite{Siegrist} that a distribution is exponential if and only if it has constant rate and
\[
F(x y) = F(x) F(y), \quad x, \, y \in S
\]
However there are constant rate distributions that are not exponential. Moreover, of course, the exponential property makes no sense for a general poset.

Constant rate distributions can be characterized in terms of the upper operator $U$ or the lower operator $L$. In the first case, the characterization is an eigenvalue condition and in the second case a moment condition.

\begin{proposition}
The poset $(S, \preceq, \lambda)$ supports a distribution with constant rate $\alpha$ if and only if there exists a strictly positive $G \in \mc{L}(S)$ with 
\begin{equation}
U(G) = \frac{1}{\alpha}G \label{eq.Eigenvalue}
\end{equation}
\end{proposition}

\begin{proof} 
If $F$ is the UPF of a distribution with constant rate $\alpha$, then trivially $F$ satisfies the conditions of the proposition since $f = \alpha F$ is a PDF with UPF $F$. Conversely, if $G \in \mc{L}(S)$ is strictly positive and satisfies (\ref{eq.Eigenvalue}), then $f := G/||G||$ is a PDF and $U(f) = \frac{1}{\alpha} f$, so the distribution with PDF $f$ has constant rate $\alpha$.
\end{proof}

\begin{proposition}
Random variable $X$ has constant rate $\alpha$ if and only if
\begin{equation}  \label{eq.rate3}
\E[Lg(X)] = \frac{1}{\alpha} \E[g(X)]
\end{equation}
for every $g \in \mc{D}(S)$.
\end{proposition}

\begin{proof} 
Suppose that $X$ has constant rate $\alpha$ and UPF $F$, so that $f = \alpha F$ is a PDF of $X$.  Let $g \in \mc{D}(S)$, From Proposition \ref{p.expect},
\[ \E[Lg(X)] = \int_S g(x) F(x) d\lambda(x) = \int_S g(x) \frac{1}{\alpha} f(x) d\lambda(x) = \frac{1}{\alpha} \E[g(X)] \]
Conversely, suppose that (\ref{eq.rate3}) holds for every $g \in \mc{D}(S)$.  Again let $F$ denote the UPF of $X$. By Proposition \ref{p.expect}, condition (\ref{eq.rate3}) is equivalent to 
\[ \int_S \alpha g(x) F(x) d\lambda(x) = \E[g(X)] \]
It follows that $f = \alpha F$ is a PDF of $X$.
\end{proof}

If $X$ has constant rate $\alpha$, then iterating (\ref{eq.rate3}) gives 
\[
\E[L^n(g)(X)] = \frac{1}{\alpha^n} \E[g(X)], \quad n \in \N
\]
In particular, if $g = \bs{1}$ then
\[ \E[\lambda_n(X)] = \frac{1}{\alpha^n}, \quad n \in \N \]
and if $n = 0$, $\E[\lambda_1(X)] = \E(\lambda(D[X])) = 1 / \alpha$. 

Suppose now that $(S, \preceq)$ is a discrete standard poset and that $X$ has constant rate $\alpha$ on $S$. From Proposition \ref{p.rate0}, $\alpha \leq 1$, and if $\alpha = 1$, all elements of $S$ are maximal, so that $(S, \preceq)$ is an antichain.  Conversely, if $(S, \preceq)$ is an antichain, then any distribution on $S$ has constant rate 1. On the other hand, if $(S, \preceq)$ is not an antichain, then $\alpha < 1$, and $S$ has no maximal elements.

Consider again a discrete standard poset $(S, \preceq)$. For $x, \, y \in S$, we say that $x$ and $y$ are {\em upper equivalent} if $I(x) = I(y)$. Upper equivalence is the equivalence relation associated with the function $s \mapsto I(s)$ from $S$ to $\mathcal{P}(S)$.  Suppose that $X$ has constant rate $\alpha$ on $S$ and UPF $F$. If $x, \, y$ are upper equivalent, then $\Pr(X \succ x) = \Pr(X \succ y)$ so from (\ref{eq.rate0}), 
\[ F(x) = \frac{\Pr(X \succ x)}{1 - \alpha} = \frac{\Pr(X \succ y)}{1 - \alpha} = F(y) \]
Thus, the UPF (and hence also the PDF) are constant on the equivalence classes.

\subsection{Ladder variables and partial products}

Assume that $(S, \preceq, \lambda)$ is a standard poset. Suppose that $\bs{X} = (X_1, X_2, \ldots)$ is an IID sequence with common UPF $F$, and let $\bs{Y} = (Y_1, Y_2, \ldots)$ be the corresponding sequence of ladder variables.

\begin{proposition} \label{p.gamma1}
If the distribution has constant rate $\alpha$ then
\begin{enumerate}
\item \label{p.gamma1.1}
$\bs{Y}$ is a homogeneous Markov chain on $S$ with transition probability density $g$ given by 
\[
g(y, z) = \alpha \frac{F(z)}{F(y)}, \quad (y, z) \in D_2
\]
\item \label{p.gamma1.2}
$(Y_1, Y_2, \ldots, Y_n)$ has PDF $g_n$ given by
\[
g_n(y_1, y_2, \ldots, y_n) = \alpha^n F(y_n), \quad (y_1, y_2, \ldots, y_n) \in D_n
\]
\item \label{p.gamma1.3}
$Y_n$ has PDF $f_n$ given by
\[
f_n(y) = \alpha^n \lambda_{n-1}(y) F(y), \quad y \in S
\]
\item \label{p.gamma1.4} 
The conditional distribution of $(Y_1, Y_2, \ldots, Y_{n-1})$ given $Y_n = y$ is uniform on $D_{n-1}[y]$.
\end{enumerate}
\end{proposition}

\begin{proof} 
Parts \ref{p.gamma1.1} and \ref{p.gamma1.2} follows immediately from Proposition \ref{p.ladder1}. Part \ref{p.gamma1.3} follows from part \ref{p.gamma1.2} and part \ref{p.gamma1.4} from parts \ref{p.gamma1.2} and \ref{p.gamma1.3}.
\end{proof}
 
When $\bs{X}$ is an IID sequence with constant rate $\alpha$, the sequence of ladder variables $\bs{Y}$ is analogous to the arrival times in the ordinary Poisson process. By Part \ref{p.gamma1.4}, this defines the most random way to put a sequence of ordered points in $S$. Part \ref{p.gamma1.4} almost characterizes constant rate distributions.

\begin{proposition} \label{p.gamma2}
Suppose that $(S, \preceq \lambda)$ is a standard, {\em connected} poset, and that $\bs{Y}$ is the sequence of ladder variables associated with an IID sequence $\bs{X}$. If the conditional distribution of $Y_1$ given $Y_2 = y$ is uniform on $D[y]$, then the common distribution of $\bs{X}$ has constant rate.
\end{proposition}

\begin{proof}
From (\ref{eq.jointPDF}), the conditional PDF of $Y_1$ given $Y_2 = y \in S$ is 
\[
h_1(x | y) = \frac{1}{C(y)} r(x), \quad x \in D[y]
\]
where $C(y)$ is the normalizing constant. But this is constant in $x\in D[y]$ by assumption, and hence $r$ is constant on $D[y]$ for each $y \in S$. Thus, it follows that $r(x) = r(y)$ whenever $x \perp y$. Since $S$ is connected, $r$ is constant on $S$.
\end{proof}

If the poset is not connected, it's easy to construct a counterexample to Proposition \ref{p.gamma2}: Consider two parallel copies of $(\N, \le)$.  Put $p$ of a geometric distribution with rate $\alpha$ on the first copy and $1 - p$ of a geometric distribution with rate $\beta$ on the second copy, where $p, \, \alpha, \, \beta \in (0, 1)$ and $\alpha \ne \beta$. The resulting distribution has rate $\alpha$ on the first copy and rate $\beta$ on the second copy. If $\bs{X}$ is an IID sequence with the distribution so constructed, and $\bs{Y}$ the corresponding sequence of ladder variables, then the conditional distribution of $(Y_1, \ldots, Y_{n-1})$ given $Y_n = y$ is uniform for each $y \in S$.

Suppose now that $\bs{X}$ is an IID sequence from a distribution with constant rate $\alpha$ and UPF $F$. Let $\bs{Y}$ denote the corresponding sequence of ladder variables, and suppose that the sequence $\bs{Y}$ is thinned with probability $p \in (0, 1)$ as describe in Section \ref{ss.ladder1}. Let $Y_M$ denote the first accepted point.

\begin{proposition} \label{p.thinned}
The PDF $g$ of $Y_M$ is is given by
\[ g(x) = p \alpha \Lambda[x, \alpha (1 - r)] F(x), \quad x \in S \]
where $\Lambda$ is the generating function associated with $(\lambda_n: n \in \N)$.
\end{proposition}

\begin{proof}
For $x \in S$,
\begin{align*}
g(x) &= \E[f_M(x)] = \sum_{n=1}^\infty p (1 - p)^{n-1} f_n(x) \\
&= \sum_{n=1}^\infty p (1 - p)^{n-1} \alpha^n \lambda_{n-1}(x) F(x) \\
&= \alpha p F(x) \sum_{n=1}^\infty [\alpha (1 - p)]^{n-1} \lambda_{n-1}(x)\\
&= \alpha p F(x) \Lambda[x, \alpha (1 - p)]
\end{align*}
\end{proof}

In general, $Y_M$ does not have constant rate, but as we will see in Section \ref{s.trees}, does have constant rate when $(S, \preceq)$ is a tree.

Suppose now that $(S, \cdot , \lambda)$ is a standard positive semigroup and that $\bs{X} = (X_1, X_2, \ldots)$ is an IID sequence.  Let $\bs{Y}$ denote the sequence of ladder variables and $\bs{Z}$ the partial product sequence.  If the underlying distribution of $\bs{X}$ is exponential, then the distribution has constant rate, so Proposition \ref{p.gamma1} applies. But by Proposition \ref{p.product1}, $\bs{Z}$ is also a homogeneous Markov chain with transition probability
\[
h(y, z) = f(y^{-1} z) = \alpha F(y^{-1} z) = \alpha \frac{F(y)}{F(z)}, \quad y \in S, z \in I[y]
\]
where $\alpha$ is the rate constant, $f$ the PDF, and $F$ the UPF. Thus $\bs{Z}$ also satisfies the results in Proposition \ref{p.gamma1}, and in particular, $\bs{Y}$ and $\bs{Z}$ are equivalent.  The converse is also true.

\begin{proposition}
If $\bs{Y}$ and $\bs{Z}$ are equivalent then the underlying distribution of $\bs{X}$ is exponential.
\end{proposition}

\begin{proof}
Let $f$ denote the common PDF of $\bs{X}$ and $F$ the corresponding UPF.  Since $Y_1 = Z_1 = X_1$, the equivalence of $\bs{Y}$ and $\bs{Z}$ means that the two Markov chains have the same transition probability density, almost surely with respect to $\lambda$.  Thus we may assume that
\[
\frac{f(z)}{F(y)} = f(y^{-1}z), \quad (y, z) \in D_2
\]
Equivalently, 
\[
f(x u) = F(x) f(u), \quad x, \, u \in S
\]
Letting $u = e$ we have $f(x) = f(e) F(x)$, so the distribution has constant rate $\alpha = f(e)$. But then we also have $\alpha F(x u) = F(x) \alpha F(u)$, so $F(x u) = F(x) F(u)$, and hence the distribution is exponential.
\end{proof}

\section{Trees} \label{s.trees}

In this section we consider a standard discrete poset $(S, \preceq)$ whose covering graph is a rooted tree. Thus the root $e$ is the minimum element. When $x \preceq y$, there is a unique path from $x$ to $y$ and we let $d(x, y)$ denote the distance from $x$ to $y$. We abbreviate $d(e,x)$ by $d(x)$. Let $A(x)$ denote the children of $x$, and more generally let
\[
A_n(x) = \{y \in S: x \preceq y, \, d(x, y) = n\}
\]
for $x \in S$ and $n \in \N$.  Thus, $A_0(x) = \{x\}$, and $\{A_n(x): n \in \N\}$ partitions $I[x]$. When $x = e$, we write $A_n$ instead of $A_n(e)$.

The only trees that correspond to positive semigroups are those for which $A(x)$ has the same cardinality for every $x \in S$.  In this case, the poset corresponds to the free semigroup on the alphabet $A(e)$, with concatenation as the semigroup operation \cite{Siegrist}.

Since there is a unique path from $e$ to $x$, it follows from (\ref{eq.cumulativeN}) that the cumulative function of order $n \in \N$ is 
\[ \lambda_n(x) = \binom{n + d(x)}{d(x)} = \binom{n + d(x)}{n}, \quad x \in S \]
By (\ref{eq.GF2}), the corresponding generating function is 
\[ \Lambda(x, t) = \frac{1}{(1 - t)^{d(x) + 1}}, \quad x \in S, \, |t| < 1 \]

\subsection{Upper Probability Functions}

Let $X$ be a random variable with values in $S$ having PDF $f$ and UPF $F$.  Then 
\begin{equation} \label{eq.treeUPF0}
F(x) = \Pr(X = x) + \sum_{y \in A(x)} \Pr(Y \succeq y) = f(x) + \sum_{y \in A(x)} F(y) 
\end{equation}
In particular, $F$ uniquely determines $f$. Moreover, we can characterize upper probability functions.

\begin{proposition} \label{p.treeUPF}
Suppose that $F: S \to (0, 1]$.  Then $F$ is the UPF of a distribution on $S$ if and only if
\begin{enumerate}
	\item $F(e) = 1$
	\item $F(x) > \sum_{y \in A(x)} F(y)$ for every $x \in S$.
	\item $\sum_{x \in A_n} F(x) \to 0$ as $n \to \infty$.
\end{enumerate}
\end{proposition}

\begin{proof}
Suppose first that $F$ is the UPF of a random variable $X$ taking values in $S$.  Then trivially $F(e) = 1$, and by (\ref{eq.treeUPF0}),
\[ F(x) - \sum_{y \in A(x)} F(y) = \Pr(X = x) > 0 \]
Next, $d(X) \ge n$ if and only if $X \succeq x$ for some $x \in A_n$. Moreover the events $\{X \succeq x\}$ are disjoint over $x \in A_n$.  Thus
\[ \Pr[d(X) \ge n] = \sum_{x \in A_n} F(x) \]
But by local finiteness, the random variable $d(X)$ (taking values in $\N$) has a proper (non-defective) distribution, so $\Pr[d(X) \ge n] \to 0$ as $n \to \infty)$.

Conversely, suppose that $F: S \to (0, 1]$ satisfies conditions (1)--(3).  Define $f$ on $S$ by
\[ f(x) = F(x) - \sum_{y \in A(x)} F(y), \quad x \in S \]
Then $f(x) > 0$ for $x \in S$ by (2).  Suppose $x \in S$ and let $m = d(x)$.  Then
\begin{align}
	\sum_{k=0}^{n-1} \sum_{y \in A_k(x)} f(y) &= \sum_{k=0}^{n-1} \sum_{y \in A_k(x)} \left[F(y) - \sum_{z \in A(y)} F(z)\right] \nonumber \\
	&= \sum_{k=0}^{n-1} \left[\sum_{y \in A_k(x)} F(y) - \sum_{y \in A_k(x)} \sum_{z \in A(y)} F(z) \right] \nonumber \\
	&= \sum_{k=0}^{n-1} \left[ \sum_{y \in A_k(x)} F(y) - \sum_{y \in A_{k+1}(x)} F(y) \right] \nonumber \\
	& = F(x) - \sum_{y \in A_n(x)} F(y) \label{eq.treeCollapse}
\end{align}
since $A_0(x) = \{x\}$ and since the sum collapses. But
\[ 0 \le \sum_{y \in A_n(x)} F(y) \le \sum_{y \in A_{m+n}} F(y) \to 0 \text{ as } n \to \infty \]
Thus letting $n \to \infty$ in (\ref{eq.treeCollapse}) we have 
\begin{equation} \label{eq.treeUPF}
\sum_{y \in I[x]} f(y) = F(x), \quad x \in S 
\end{equation}
Letting $x = e$ in (\ref{eq.treeUPF}) gives $\sum_{y \in S} f(y) = 1$ so $f$ is a PDF on $S$. Another application of (\ref{eq.treeUPF}) then shows that $F$ is the UPF of $f$.
\end{proof} 

Note that $(S, \preceq)$ is a lower semi-lattice. Hence if $X$ and $Y$ are independent random variables with values in $S$, with UPFs $F$ and $G$, respectively, then $X \wedge Y$ has UPF $FG$.

\begin{proposition} \label{p.treeUPF2}
Suppose that $F: S \to (0, 1]$ satisfies $F(e) = 1$ and 
\begin{equation} \label{eq.weak}
F(x)\geq \sum_{y \in A(x)} F(y), \quad x \in S 
\end{equation}
For $p \in (0, 1)$ and define $F_p: S \to (0, 1]$ by $F_p(x) = p^{d(x)}F(x)$. Then $F_p$ is an UPF on $S$.
\end{proposition}

\begin{proof}
First, $F_p(e) = p^0 F(e) = 1$.  Next, for $x \in S$, 
\begin{align*}
	\sum_{y \in A(x)} F_p(y) &= \sum_{y \in A(x)} p^{d(y)} F(y) = p^{d(x)+1} \sum_{y \in A(x)} F(x) \\
	&\le p^{d(x) + 1} F(x) < p^{d(x)} F(x) = F_p(x)
\end{align*}
A simple induction using (\ref{eq.weak}) shows that $\sum_{x \in A_n} F(x) \le F(e) = 1$ for $n \in \N$ so
\[
\sum_{x \in A_n} F_p(x) = \sum_{x \in A_n} p^{d(x)} F(x) = p^n \sum_{x \in A_n} F(x) \le p^n \to 0 \text{ as } n \to \infty
\]
so it follows from Proposition \ref{p.treeUPF} that $F_p$ is an UPF.
\end{proof}

Note that $x \mapsto p^{d(x)}$ is not itself an UPF, unless the tree is a path, since properties 2 and 3 in Proposition \ref{p.treeUPF} will fail in general. Thus, even when $F$ is an UPF, we cannot view $F_p$ simply as the product of two UPFs in general.  However, We can give a probabilistic interpretation of the construction in Proposition \ref{p.treeUPF2} in this case.  Thus, suppose that $X$ is a random variable taking values in $S$ with UPF $F$ and PDF $f$.  Moreover, suppose that each edge in the tree $(S, \preceq)$, independently of the other edges, is either {\em working} with probability $p$ or {\em failed} with probability $1 - p$. Define $U$ by
\[
U = \max\{u \preceq X: \text{ the path from } e \text{ to } u \text{ is working} \}
\]

\begin{corollary} \label{c.treeUPF}
Random variable $U$ has UPF  $F_p$ given in Proposition \ref{p.treeUPF2}.
\end{corollary}

\begin{proof}
If $X = x$ and $u \preceq x$, then $U \succeq u$ if and only if the path from $e$ to $u$ is working. Hence $\Pr(U \succeq u | X = x) = p^{d(x)}$ for $x \in S$ and $u \preceq x$. So conditioning on $X$ gives
\[
\Pr(U \succeq u) = \sum_{x \succeq u} p^{d(u)} f(x) = p^{d(u)} F(u)
\]
\end{proof}

\subsection{Rate Functions}\label{ss.treesRate}

Next we are interested in characterizing rate functions of distributions that have support $S$.  If $r$ is such a function, then as noted earlier, $0 < r(x) \le 1$ and $r(x) = 1$ if and only if $x$ is a leaf.  Moreover, if $F$ is the UPF, then $F(e) = 1$ and
\begin{equation} \label{eq.treeRate0}
	\sum_{y \in A(x)} F(y) = [1 - r(x)]F(x)
\end{equation}
Conversely, these conditions give a recursive procedure for constructing an UPF corresponding to a given rate function.  Specifically, suppose that $r: S \to (0, 1]$ and that $r(x) = 1$ for every leaf $x \in S$. First, we define $F(e) = 1$. Then if $F(x)$ has been defined for some $x \in S$ and $x$ is not a leaf, then we define $F(y)$ for $y \in A(x)$ arbitrarily, subject only to the requirement that $F(y) > 0$ and that (\ref{eq.treeRate0}) holds. Note that $F$ satisfies the first two conditions in Proposition \ref{p.treeUPF}.  Hence if $F$ satisfies the third condition, then $F$ is the UPF of a distribution with support $S$ and with the given rate function $r$. The following proposition gives a simple sufficient condition.

\begin{proposition} \label{p.treeRate}
Suppose that $r : S \to (0, 1]$ and that $r(x) = 1$ for each leaf $x \in S$.  If there exists $\alpha > 0$ such that $r(x) \ge \alpha$ for all $x \in S$, then $r$ is the rate function of a distribution with support $S$.
\end{proposition}

\begin{proof}
Let $F: S \to (0, 1]$ be any function constructed according to the recursive procedure above.  Then as noted above, $F$ satisfies the first two conditions in Proposition \ref{p.treeUPF}. A simple induction on $n$ shows that
\begin{equation} \label{eq.treeRate}
\sum_{x \in A_n} F(x) \le (1 - \alpha)^n, \quad n \in \N
\end{equation}
so the third condition in Proposition \ref{p.treeUPF} holds as well.
\end{proof}

Condition (\ref{eq.treeRate}) means that the distribution of $d(X)$ is stochastically smaller than the geometric distribution on $\N$ with rate constant $\alpha$. If $(S, \preceq)$ is not a path, then the rate function does not uniquely determine the distribution.  Indeed, if $x$ has two or more children, then there are infinitely many ways to satisfy (\ref{eq.treeRate0}) given $F(x)$.

\subsection{Constant rate distributions}

Recall that if $(S, \preceq)$ has maximal elements (leaves) then there are no constant rate distribution with support $S$, except in the trivial case that $(S, \preceq)$ is an antichain.

\begin{corollary} \label{c.treeRate}
Suppose that $(S, \preceq)$ is a rooted tree without leaves. Then $F: S \to (0, 1]$ is the UPF of a distribution with constant rate $\alpha$ if and only if $F(e) = 1$ and
\begin{equation} \label{eq.tree}
\sum_{y \in A(x)} F(y) = (1 - \alpha) F(x), \quad x \in S
\end{equation}
\end{corollary}

\begin{proof}
This follows immediately from Proposition \ref{p.treeRate}
\end{proof}

\begin{corollary}
Suppose that $X$ has constant rate $\alpha$ on $(S, \preceq)$. Then $d(X)$ has the geometric distribution on $\N$ with rate $\alpha$.
\end{corollary}

\begin{proof}
For $n \in \N$,
\[
\Pr[d(X) \ge n] = \sum_{x \in A_n} \Pr(X \succeq x) = \sum_{x \in A_n} F(x) = (1 - \alpha)^n
\]
\end{proof}

As a special case of the comments in Section \ref{ss.treesRate}, we can construct the UPFs of constant rate distributions on $(S, \preceq)$ recursively: Start with $F(e) = 1$. If $F(x)$ is defined for a given $x \in S$, then define $F(y)$ for $y \in A(x)$ arbitrarily, subject only to the conditions $F(y) > 0$ and that (\ref{eq.tree}) holds. 

\subsection{Ladder Variables and the point process}

Let $F$ be the UPF of a distribution with constant rate $\alpha$, so that $F$ satisfies the conditions in Corollary \ref{c.treeRate}. Let $\bs{X} = (X_1, X_2, \ldots)$ be an IID sequence with common distribution $F$ and let $\bs{Y} = (Y_1, Y_2, \ldots)$ be the corresponding sequence of ladder variables. By Proposition \ref{p.gamma1}, the distribution of $Y_n$ has PDF
\[
f_n(x) = \alpha^n \binom{n + d(x) - 1}{d(x)} F(x), \quad x \in S
\]

Consider now the thinned point process associated with $\bs{Y}$, where a point is accepted with probability $p$ and rejected with probability $1 - p$, independently from point to point.  

\begin{proposition} \label{p.thinned2}
The distribution of the first accepted point has constant rate $p \alpha / (1 - \alpha + p \alpha)$.
\end{proposition}

\begin{proof}
By Proposition \ref{p.thinned}, the PDF of the first accepted point is
\begin{align*}
g(x) &= p \alpha \Lambda(x, (1 - p) \alpha) F(x) = p \alpha \frac{1}{[1 - (1 - p)\alpha]^{d(x) + 1}} F(x)\\
&= \frac{r \alpha}{1 - \alpha + p \alpha} \frac{F(x)}{(1 - \alpha + p \alpha)^{d(x)}}, \quad x \in S
\end{align*}
Consider the function $G: S \to (0, 1]$ given by 
\begin{equation*}
G(x) = \frac{F(x)}{(1 - \alpha + p \alpha)^{d(x)}}, \quad x \in S
\end{equation*}
Note that $G(e) = 1$ and for $x \in S$
\begin{align*}
\sum_{y \in A(x)} G(y) &= \sum_{y \in A(x)} \frac{F(y)}{(1 - \alpha + p \alpha)^{d(y)}}\\
&= \frac{1}{(1 - \alpha + p \alpha)^{d(x) + 1}} \sum_{y \in A(x)} F(y)\\
&= \frac{1 - \alpha}{(1 - \alpha + p \alpha)^{d(x) + 1}} F(x)\\
&= \frac{1 - \alpha}{1 - \alpha + p \alpha} G(x)
\end{align*}
\end{proof}

In Proposition \ref{p.thinned2} The UPF $F$ is related to the UPF $G$ by the construction in Corollary \ref{c.treeUPF}.  That is, suppose $Y$ denotes the first accepted point in the thinned process.  Then the basic random variable $X$ that we started with can be constructed as
\[
X = \max\{x \preceq Y: \text{ there is a working path from } e \text{ to } x\}
\]
where each edge is working, independently, with probability $1 - \alpha + p \alpha$.

\section{Other Posets}

\subsection{Basic Constructions}

In this section, we mention results for a few basic constructions. For the most part, the arguments are straightforward, so the details are omitted.

On a standard poset, a mixture of constant rate distributions, each with rate $\alpha$, also has constant rate $\alpha$.  

Suppose that $X$ and $Y$ are independent random variables on (possiby different) standard posets. Then $(X, Y)$ has constant rate on the product poset (with the product measure) if and only $X$ and $Y$ have constant rates. Moreover, in this case, if $X$ has rate $\alpha$, $Y$ has rate $\beta$ then $(X, Y)$ has rate $\alpha \beta$. If $X$ and $Y$ are not independent, then it is possible for $(X, Y)$ to have constant rate but not $X$ and $Y$ individually.  It's easy to construct examples using mixtures.  On $[0, \infty)^n$ (with product order and measure), the only distributions with constant rate are mixtures of constant rate distributions that correspond to independent components \cite{Puri}.  This result fails for general standard posets, even those associated with positive semigroups.

Suppose that $X$ has constant rate on a standard poset and that $Y$ is uniformly distributed on a compact standard poset, and that $X$ and $Y$ are independent. Then $(X, Y)$ has constant rate on the lexicographic product of the two posets.

\subsection{Finite subsets of $\N_+$} \label{ss.subsets}

Let $S$ denote the collection of finite subsets of $\N_+$, partially ordered by set inclusion.  From \cite{Siegrist3} the poset $(S, \subseteq)$ is associated with a positive semigroup; the semigroup operation is 
\[
x y = x \cup \{x^c(i): i \in y \}
\]
where $x^c(i)$ is the $i$th element of $x^c$ (ordered by $\leq$ on $\N_+$).  The cumulative function of order $n \in \N$ is 
\begin{equation*}
\lambda_n(x) = (n + 1)^{\#(x)}, \quad x \in S
\end{equation*}
where $\#(x)$ denotes the number of elements in $x$.

The poset $(S, \subseteq)$ is universal in the sense that every discrete, standard poset is isomorphic to a sub-poset of $(S, \subseteq)$. To see this, suppose that $(T, \preceq)$ is such a poset where, without loss of generality, we assume that $T$ is a subset of $\N_+$.  Let $\hat T = \{D[x]: x \in T\}$.  Then $x \mapsto D[x]$ is an order isomorphism from $(T, \preceq)$ to $(\hat T, \subseteq)$, and of course from our basic assumptions, $D[x]$ is finite for each $x \in T$. Thus $(\hat T, \subseteq)$ is a sub-poset of $(S, \subseteq)$.

Suppose that $X$ has PDF $f$ and UPF $F$.  Let  
\begin{align*}
	A_n &= \{x \in S: \#(x) = n\}\\
	A_n(x) &= \{y \in S: x \subseteq y, \, \#(y) = \#(x) + n\}
\end{align*}
So $\{A_n: n \in \N\}$ is a partition of $S$ and $\{A_n(x): n \in \N\}$ is a partition of $I[x]$. The UPF $F$ determines the PDF $f$ as follows:
\[ f(x) = \sum_{n = 0}^\infty (-1)^n \sum_{y \in A_n(x)} F(y), \quad x \in S \]

From \cite{Siegrist3}, $(S, \cdot)$ has no exponential distributions. We have been unable to determine if $(S, \subseteq)$ has constant rate distributions, but we can show that if there {\em is} a constant rate distribution, then the number of elements in the random set must have a Poisson distribution.

\begin{lemma} \label{l.subsets}
Suppose that $X$ has constant rate $\alpha$ and let $U = \#(X)$. Then
\[
\alpha \Pr(U = k) = \E\left[(-1)^{U+k} \binom{U}{k}\right]
\]
\end{lemma}

\begin{proof}
If $X$ has constant rate $\alpha$ then
\begin{equation} \label{eq.subsets2}
f(x) = \frac{1}{\alpha} \sum_{n = 0}^\infty (-1)^n \sum_{y \in A_n(x)} f(y), \quad x \in S
\end{equation}
From (\ref{eq.subsets2})
\[
\Pr(U = k) = \sum_{x \in S_k} f(x) = \frac{1}{\alpha} \sum_{n = 0}^\infty (-1)^n \sum_{x \in S_k} \sum_{y \in S_n(x)} f(y)
\]
The last two sums are over all $x, y \in S$ with $\#(y) = n + k$, and $x \subseteq y$.  Interchanging the order of summation gives
\begin{align*}
\Pr(U = k) &= \frac{1}{\alpha} \sum_{n=0}^\infty (-1)^n \sum_{y \in S_{n+k}} \sum \{f(y): x \in S_k, \, x \subseteq y\} \\
&= \frac{1}{\alpha} \sum_{n=0}^\infty (-1)^n \sum_{y \in S_{n+k}} \binom{n+k}{k} f(y)\\
&= \frac{1}{\alpha} \sum_{n=0}^\infty (-1)^n \binom{n+k}{k} \Pr(U = n + k)
\end{align*}
Equivalently (with the substitution $j = n + k$),
\[
\alpha \Pr(U = k) = \sum_{j = k}^\infty (-1)^{j-k} \binom{j}{k} \Pr(U = j)
\]
With the usual convention on binomial coefficients ($\binom{a}{b} = 0$ if $b > a$) we have
\[
\alpha \Pr(U = k) = \sum_{j=0}^\infty (-1)^{j+k} \binom{j}{k} \Pr(U = j) = \E\left[(-1)^{k + U} \binom{U}{k} \right]
\]
\end{proof}

\begin{corollary} \label{c.subsets}
The condition in Lemma \ref{l.subsets} is equivalent to each of the following:
\begin{enumerate}
	\item $G(t - n) = \alpha^n G(t)$ for $t \in \R$ and $n \in \N$, where $G$ is the probability generating function of $U$.
	\item $\Pr(U = n) = \alpha \E \left[\binom{U}{n}\right]$ for $n \in \N$.
	\item $(-1)^n \E\left[(-1)^U \binom{U}{n}\right] = \alpha^2 \E\left[\binom{U}{n}\right]$ for $n \in \N$.
\end{enumerate}
\end{corollary}

\begin{proposition}
The condition in Lemma \ref{l.subsets}, and hence also Corollary \ref{c.subsets}, characterize the Poisson distribution with parameter $\mu = -\ln(\alpha)$.
\end{proposition}

\begin{proof}
It is straightforward to see that the conditions are equivalent to the famous Rao-Rubin \cite{Rao} characterization of the Poisson distribution.
\end{proof}

\subsection{Open Questions}

We end with two open questions:  
\begin{enumerate}
\item Does the poset $(S, \subseteq)$ of Section \ref{ss.subsets} support a constant rate distribution?
\item More generally, does every standard discrete poset without maximal elements support a constant rate distribution?
\end{enumerate}

\end{document}